\documentclass[12pt, a4paper]{article}
\usepackage{amsmath,amscd}
\usepackage{amstext}
\usepackage{amssymb,epsfig}
\usepackage{amsthm}
\usepackage{amsfonts}
\usepackage{graphicx}
\usepackage{psfrag}
\usepackage[all]{xy}
\usepackage{tikz}
\usetikzlibrary{arrows,automata}

\usepackage{hyperref}

\renewcommand{\footnote}{\endnote}

\newtheorem{theorem}{Theorem}

\newtheorem{lemma}[theorem]{Lemma}
\newtheorem{proposition}[theorem]{Proposition}

\newtheorem{corollary}[theorem]{Corollary}

\theoremstyle{definition}
\newtheorem{definition}[theorem]{Definition}
\newtheorem{remark}[theorem]{Remark}

\newtheorem{example}[theorem]{Example}

\usepackage[utf8]{inputenc}      

\begin{document}


\title{Symbolic dynamics for 
piecewise rotations:\\ 
The case of the bijective symmetric 
maps }

\author{Nicolas B\'edaride\footnote{ Aix Marseille Université, CNRS, Centrale Marseille, I2M, UMR 7373, 13453 Marseille, France. Email: nicolas.bedaride@univ-amu.fr}\and  Idrissa Kabor\'e\footnote{Universit\'e Nazi 
Boni, 
UFR sciences, Bobo-Dioulasso 01 PB 1091, Burkina Faso. Email: ikaborei@yahoo.fr}}

\date{}

\maketitle
\begin{abstract}
We consider a specific 
class of piecewise rotations of the plane that are continuous on two half-planes, 
as studied in \cite{Bosh.Goet.03}, \cite{Goet.Quas.09} and \cite{Che.Goe.Qua.12}. 
Assuming that 
the angle belongs to the set $\{\frac{\pi}{2},\frac{\pi}{3},\frac{\pi}{6},\frac{\pi}{4}\}
$, 
we give a description of the symbolic dynamics of this map in the bijective symmetric case.

\end{abstract}

\section{Introduction}
In this paper we consider the dynamics of a 
particular class of piecewise isometries.
A piecewise isometry 
of 
$\mathbb{R}^n$ is defined in the following way: consider a finite set of hyperplanes, and let $X$ be the complement of their union. It  has several connected components. The piecewise isometry is a map $T$ from $X$ to $\mathbb{R}^n$ which is 
defined on each connected component as the restriction of an isometry of $\mathbb{R}^n$. Now consider the 
preimages 
of the union of the hyperplanes by $T$: it is a set of Lebesgue measure zero. Thus, almost every point of $X$ has an orbit under 
$T$, and the resulting dynamical system $(X,T)$ has generated some natural interest since more than 20 years.

The class of such maps has been well studied in dimension 
$1$, 
with primary example given by interval exchange maps: in this case the map is bijective, equal to the identity outside of a compact interval, and the isometries which locally define $T$ are translations. 
More recently the paper \cite{Ad.Ki.Tr.01} has appeared, where for the first time such a system in dimension 2 has been considered.
Since 
then, 
different examples 
in dimension 2 
have been worked 
out, 
in order to exhibit different types of 
behavior, 
see for example \cite{Goet.98} or \cite{As.Go.04}. The first general result has been obtained by 
Buzzi, 
who proved that every piecewise isometry has zero entropy, see \cite{Buzz.01}. An important class of piecewise isometries is the outer billiard, also called dual billiard. Around this map a lot of developments 
have taken place in recent years, 
most prominently through  
the work of 
Schwartz, see \cite{Schw.09}, \cite{Schwartz.10} and \cite{Schwartz.12}. 
In particular he 
describes the first example of a piecewise isometry of the plane which contains 
both at the same time, 
an unbounded orbit and periodic orbits.

In this paper 
we study 
a still different 
example, which was 
first 
introduced 
by Boshernitzan and Goetz 
in \cite{Bosh.Goet.03}. 
This map is called a piecewise rotation. Up to date it is perhaps the piecewise isometry which has 
received most attention, 
see \cite{BK2}, \cite{Goet.Quas.09} and \cite{Che.Goe.Qua.12}. Consider a line in the plane 
(identified for simplicity with the complex plane $\mathbb C$, while the line is assumed to be the real axis $\mathbb{R}$),
and fix 
two points $S_0, S_1$ outside the line 
as well as 
an angle $2\pi\theta\in [0,2\pi)$. 
The map is 
defined on each half plane 
of $\mathbb {C} \smallsetminus \mathbb{R}$ 
by a rotation around $S_i$ with angle $2\pi\theta$.  The phase space of 
this 
map can be described 
by 
two parameters: one for the 
angle, 
and one which measures the relative 
position 
of the centers. If the middle of the segment $[S_0,S_1]$ belongs to the real axis, then this 
second 
parameter is 
zero, 
and the map is called symmetric.
According to the position 
of the centers of rotations 
$S_i$ 
this map can be bijective, 
non-injective or non-surjective. 
In \cite{Bosh.Goet.03} Boshernitzan and Goetz show that in the 
last two 
cases the map is either globally attractive or globally repulsive. In the bijective case, Goetz and Quas have shown that for a rational angle every orbit is bounded, see \cite{Goet.Quas.09}. 
In order to prove this 
result, 
they introduce symbolic dynamics for this 
map, based of 
the notion of 
``rotationally coded'' 
points.

In the present paper we want to give a 
specific 
description of the symbolic dynamics of this map. 
In order to achieve a higher degree of precision 
than the previous results, we restrict our study to a finite family of 
angles, 
and to the symmetric bijective case. 
The 
non-symmetric bijective case is dealt with 
in another paper, see \cite{BK2}.
In the 
case of $\frac{\pi}{4}$ 
we 
exhibit explicitly certain 
bounded orbits which are not periodic, see Section \ref{18}. Thus these orbits do not come from rotationally coded orbits.  
Our method of investigation is close to the one introduced in \cite{Bed.Ca.11} for the outer billiard outside regular polygons. The main idea 
here 
is to find a reasonable set, 
such that one 
can consider the first return map and prove that it is 
conjugate 
to the initial map. This allows us to use substitutions in order to describe the language of the map.
Of course, in most of the cases, the approach is not 
quite as 
simple: 
one needs to introduce several 
different transformations before being able to find a good renormalisation, see Propositions \ref{lien-dual} and \ref{prop-detail} for a complete study of one particular case.

\medskip

{\em This work has been supported by the Agence Nationale de la Recherche -- ANR-10-JCJC 01010.}

\section{Definition of a piecewise isometry and some background}
\subsection{Definitions}
We refer to Figure \ref{fig1}:
Consider a line $l$ in $\mathbb{R}^2$, it splits the plane on two half-planes. Now we define a piecewise isometry $T$ on $\mathbb{R}^2$ such that the restriction to each half-plane is given by a rotation. The two rotations are of the same angle (denoted $2\pi\theta$ with $\theta\in [0,1])$ with different centers. We also assume that the centers of rotation are not on the line $l$. 
Without loss of generality we can identify the plane with the complex numbers $\mathbb{C}$ and the line with the real line $\mathbb{R}$. If the centers have coordinates $z_0$ and $z_1$, the map is given by:
$$\begin{array}{ccc}
\mathbb{C}\setminus\mathbb{R}&\rightarrow&\mathbb{C}\\
z&\mapsto&T(z)=\begin{cases}e^{2i\pi \theta}(z-z_0)+z_0\quad\text{if}\quad Im(z)> 0\\ e^{2i\pi \theta}(z-z_1)+z_1\quad\text{if}\quad Im(z)<0 \end{cases}
\end{array}
$$
\begin{remark}
 Consider the set of complex numbers $z$ such that there exists an integer $n$ with $T^nz\in\mathbb{R}$. This set of points is called the {\bf set of discontinuity points}. It is included in a countable union of lines (the backward images of $\mathbb R$ by some power of the rotations), thus it is of Lebesgue measure zero . Outside this set, for  every point $z$ we can define $T^nz$ for $n\in\mathbb N$. In the sequel we will consider the map $T$ restricted to this set of full Lebesgue measure.  
\end{remark}

Now consider the two images of $\mathbb R$ by the two rotations. We obtain two parallel lines. The map $T$ is bijective if these lines coincide. If not, then the map is either non injective or non surjective.
The case where $T$ is bijective has been studied by Goetz and Quas, see \cite{Goet.Quas.09}. In the bijective case, the map can be written in a particularly easy way, up to a scaling map:
$$T(z)=\begin{cases}e^{2i\pi \theta}(z+s+1)\quad\text{if}\quad Im(z)> 0\\ e^{2i\pi \theta}(z+s-1)\quad\text{if}\quad Im(z)<0 \end{cases}$$
where $s$ is a real number. If $s=0$ the map is called a {\bf symmetric map}. Note that if $s=0$, then the two centers of rotations are centrally symmetric with respect to the origin, see Equation \ref{centres}. 

By a slight abuse of notation the parameter $\theta$ is called the angle of the map.

\begin{figure}
\begin{tikzpicture}[scale=.4]
\draw (-10,0)--(10,0);
\draw[fill] (-2,1) circle(.15);
\draw[fill] (0,-1) circle(.15);
\draw[] (1,3) circle(.15);
\draw[] (-4,4) circle(.15);
\draw[] (-5,-1) circle(.15);
\draw[] (0,-6) circle(.15);

\draw (1,3) node[right]{$m$};
\draw (-4,4) node[right]{$Tm$};
\draw (-5,-1) node[below]{$T^2m$};
\draw (0,-6) node[right]{$T^3m$};
\draw (-2,1) node[below]{$S_0$};
\draw (0,-1) node[right]{$S_1$};

\draw[red] (-1,5)--(-1,-5);
\draw[dashed] (-2,1)--(1,3);
\draw[dashed] (-2,1)--(-4,4);
\draw[dashed] (-2,1)--(-5,-1);
\draw[dashed] (0,-1)--(-5,-1);
\draw[dashed] (0,-1)--(0,-6);
\end{tikzpicture}

\caption{A piecewise rotation of angle $\frac{1}{4}$ and the four first points in the orbit of $m$. In red the image of $\mathbb R$ by the two rotations.}\label{fig1}
\end{figure}
\subsection{Coding of the map}

Let $\mathcal{A}$ be a finite set called  alphabet, a {\bf word} is a finite sequence of elements in $\mathcal{A}$, its length is the number of elements in the sequence. The set of all finite words over $\mathcal{A}$ is denoted $\mathcal{A}^*$, and $\varepsilon$ is the empty word. A (one-sided) infinite sequence of elements of $\mathcal{A}$, $u=(u_n)_{n\in\mathbb{N}}$, is called an infinite word. 
The infinite word $u$ is {\bf periodic} if there exists a finite word $v_0\dots v_n$ such that $u=v_0\dots v_n v_0\dots v_n v_0\dots v_n\dots$ Such an infinite word is denoted $v^\omega$. 
A word $v_0\dots v_k$ {\bf appears} in the word $u$ if there exists an integer $i$ such that $u_{i}\dots u_{i+k}=v_0\dots v_{k}$. In this case we say that $v$ is a {\bf factor} of $u$. For an infinite word $u$, the {\bf language} of $u$ (or the language of length $n$ respectively)  is the set of all words (or all words of length $n$ respectively) in $\mathcal{A}^*$ which appear in $u$. We denote it by $L(u)$ (or $L_{n}(u)$ respectively).

Let $P_0,P_1$ be the two half-planes bounded by the discontinuity line of $T$. 
Let $\phi:\mathbb{C}\setminus\mathbb R\mapsto \{0;1\}^\mathbb{N}$ be {\bf the coding map} where the image of a complex number $z$ is given by $\phi(z)=(u_n)_{n\in\mathbb{N}}$ such that $T^n(z)\in P_{u_n}$ for any integer $n$. 
The image under the coding map of the points which have well defined orbit defines a language by looking at factors of infinite words. 
For an infinite word $u$ in this language, a {\bf cell} is the set of points which are coded by this word:
$C_u=\{z\in \mathbb{C}\setminus \mathbb R, \phi(z)=u\}$.

Remark that the coding map could also be defined if the piecewise isometry was defined on more than two subspaces.
\begin{figure}[]
\begin{center}
\includegraphics[width=8cm]{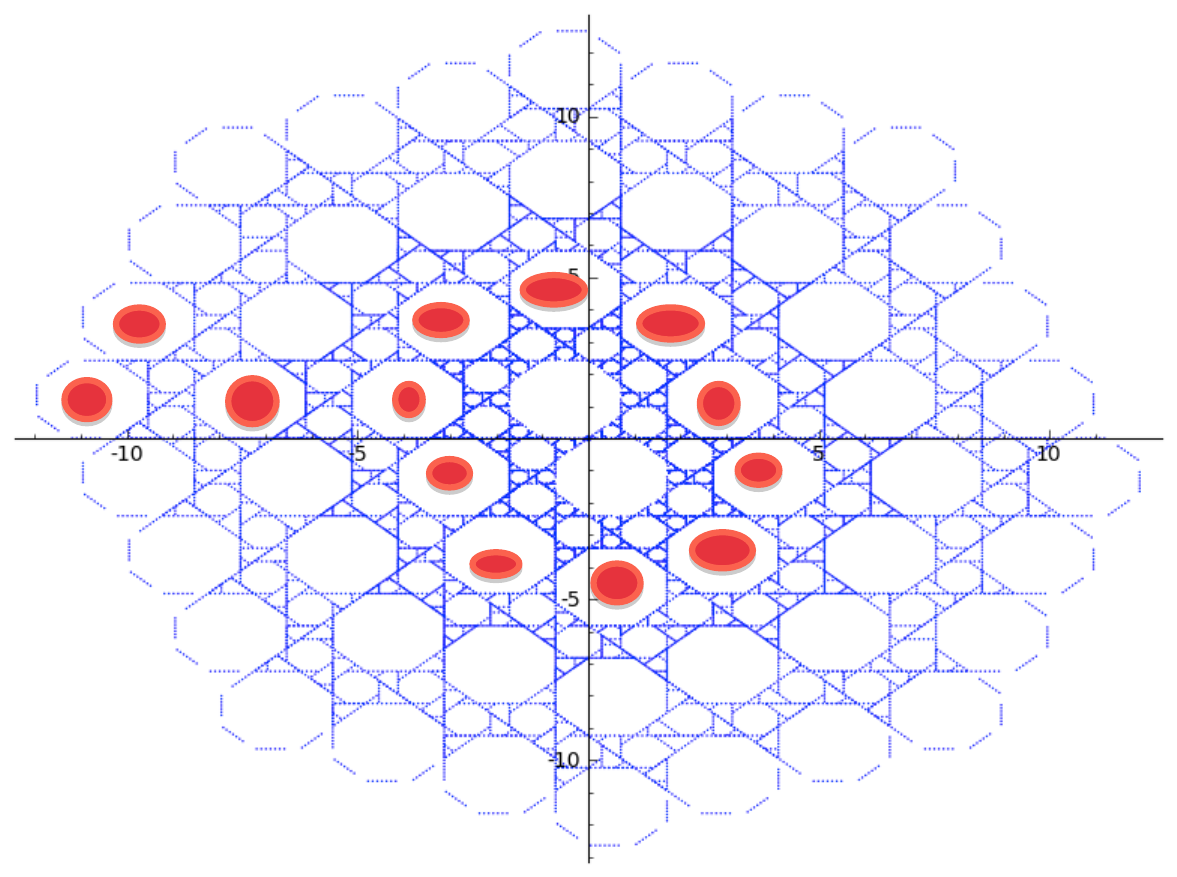}
\caption{The cellular decomposition for the angle $\theta=\frac{1}{8}$ and $s=0$.}\label{fig:tout-8}
\end{center}
\end{figure}

\section{Words and symbolic dynamics}
\subsection{Substitutive languages} 
A {\bf substitution} $\sigma$ is a map from an alphabet $\mathcal{A}$ to the set $\mathcal{A}^*\setminus\{\varepsilon\}$ of non-empty finite words on $\mathcal{A}$. It extends to a morphism of $\mathcal{A}^*$ by concatenation, that is $\sigma(uv)=\sigma(u)\sigma(v)$ for all finite words $u, v$. In the following we will denote a substitution by an array which has in the first line the elements of the alphabet and in the second line their images under the substitution, see $\sigma_4$ on Subsection \ref{sec-per} for example.

Now we introduce the notion of a {\bf substitutive language}. Consider a finite set $S$ of substitutions and  a subset $P\subset<S>$ of the monoid generated by them. Let $X$ be a finite set of words, then a substitutive language is the set of factors of elements of a subset of $\{g(u), g\in P, u\in X\}$. Of course some conditions are needed on $X$ and $P$ to be sure to define a language.
Below, we will specify $P$ and $X$ with the help of a finite oriented graph such that each edge is marked by a substitution of $S$ and on each vertex is labeled by a finite union of words of $X$.
 
We attach to each path on this graph a set of words by the following way: consider one edge marked by $\sigma$ between the vertices $A$ and $B$. Then we associate to this path the set of words $\{\sigma(a)| a\in A\}\cup\{a\in A\}.$ Furthermore to each oriented path obtained by concatenation of several edges we apply the composition of the substitutions to the starting set $A$. Consider the set of all paths on this graph with arbitrary starting vertex. The set $P\subset<S>$ is equal to the concatenation of the substitutions associated to each edge. 

By convention, if there is only one substitution, we will only draw one vertex labelled with a finite set of words. 

\begin{example}
Consider the following graph.
\begin{center}
\begin{tikzpicture}[->,>=stealth',shorten >=1pt,auto,node distance=2cm,
                    semithick]
  \tikzstyle{every state}=[fill=gray,draw=none,text=white]

 \node[state] (A)                    {$X$};
  \node[state] (B) [ right of=A] {X};
    \node[state]  (C) [ right of=B]{X};

  \path (A) edge              node {$\sigma_2$} (B)
                  edge [loop left] node {$\sigma_1$} (A)
           (B)  edge [loop above] node {$\sigma_2$} (B)
                  edge              node {$\sigma_3$} (C)
           (C) edge [loop right] node {$\sigma_3$} (C);
         
\end{tikzpicture}
\end{center}
If $X$ is a set of finite words, the language is the set of factors of the words in the set:
$$Z=\bigcup_{n,m,p\in\mathbb{N}}\{\sigma_3^p\circ\sigma_2^m\circ\sigma_1^n(X)\}.$$

\end{example}
 
 Substitutive languages will be useful in the statement of the results. They already appear for the description of the language of the outer billiard outside a regular pentagon see \cite{Bed.Ca.11}, as well as in the study of $S$-adic systems, see \cite{Bert.Dele.14} for a recent survey.
 
\subsection{Induction and substitution}
In this subsection we start from a general dynamical system $(X,T)$. We recall the definition of the induction of $T$ on a subset. Then we recall some usual facts about the relation between the codings of the two applications. 

Let $X$ be an open subset of $\mathbb{R}^2$, and $T$ a piecewise isometry from $X$ to $X$. Assume there is a partition of $X$ by the sets $U_i$ with 
$i=1,\dots, k$, and consider the language $L_U$ obtained by the associated coding map. Now assume that the first return map of $T$ on $U_1$, denoted by $T_{U_1}$, is well defined for $x\in U_1$ : $$T_{U_1}(x)=T^{n_x}(x),$$ where $n_x$ is the smallest positive integer $n$ such that $T^nx$ belongs to $U_1$ for $x\in U_1$.
Consider the measurable partition $\bigcup_{l\in I} U_1(l)$ of $U_1$ for which the return time $n_{x}$ is constant on each class $U_1(l)$ of the partition. We associate a coding map to the action of $T_{U_1}$ on this partition. We denote the associated language by $L_{U_1}$.

\begin{lemma}\label{leminducsub}
Assume that there exists a bijection $h$ between $X$ and $U_1$ which maps each $U_i$ for $i=1,\dots, k$ to one of the sets $U_1(l)$ such that for any $x\in X$ we have $h^{-1}\circ T_{U_1}\circ h(x)=T(x)$.
Then there exists a substitution $\alpha_{U_1}$ such that
$\alpha_{U_1}(L_{U_1})= L_U$. 

\end{lemma}
\begin{proof}
By hypothesis on $h$ we deduce that the partition of $U_1$ and the partition of $X$ have the same number $k$ of elements.
Thus we denote the elements of the alphabets of the languages of $T_{U_1}$ and $T$ by the same letters.

Consider $x\in X$. Then denote its coding word $\phi(x)$. By definition $h(x)$ belongs to $U_1$. The coding $\phi(h(x))$ of this point begins by $1$ (the letter associated to $U_1$). By definition, the return time of $h(x)$ to $U_1$ is bounded, thus the sequence contains an infinite number of $1's$. Now consider all the positions of $1$ in this sequence and the finite words between two consecutive $1$'s. By assumption the number of return times is finite, thus these words form a finite set and this set is  in bijection with the partition $\{U_1(l), l\}$.  
If $x$ is in one cell of the partition of $U_1$ then denote by $v_i$ the finite word  of length $n_x$ which codes the orbit $T^jx, j=0\dots n_x-1$ where $n_x$ is the return time of $x$ in $U_1$. By definition $n_x$ is independant of the choice of $x\in U_1(i)$. Now define the substitution $\alpha_{U_1}:L_U\rightarrow L_{U_1}$ by 
$$\begin{array}{|c|c|c|c|}
\hline

i&1&\dots&k\\
\hline
\alpha_{U_1}(i)&v_1&\dots& v_k\\
\hline
\end{array}$$
By definition of the substitution $\alpha_{U_1}$, we have $\phi(h(x))=\alpha_{U_1}(\phi(x))$. We deduce that $\alpha_{U_1}(L_{U})\subset L_{U_1}$. The other inclusion is clear by assumption on $h$.


\end{proof}

If the hypothesis of Lemma \ref{leminducsub} is fulfilled we say that the map has a {\bf renormalization} or is self-similar.
This lemma will be used in the different proofs in order to describe the language in terms of substitutions. 
This notion is often used in another context for the description of the symbolic dynamics of a map. 
The words $\alpha_{U_1}(i)$ are called {\bf return words}, since they correspond to the coding of points between the return times.

\section{Link with the dual billiard map}
We recall some classical facts about dual billiard map, also called outer billiard map.
We consider a convex polygon $P$ in $\mathbb{R}^2$ with $n$ vertices. The dual billiard map is a self-map of $\mathbb R^2\setminus P$. Given $m\in \mathbb R^2\setminus P$, one defines $Tm$ so that the segment $[m,Tm]$ is tangent to $P$ at its midpoint and $P$ lies to the right of the ray $\overrightarrow{m Tm}$. The map $m\mapsto Tm$ is called the {\bf dual billiard map}.

\begin{figure}
\begin{center}
\begin{tikzpicture}
\fill (0,0)--(2,0)--(1,2)--(0,1)--cycle;
\draw[dashed] (5,2)--(-1,2);
\draw[dashed] (-1,2)--(1,-2);
\draw (3,2) node{$\bullet$} node[above]{$TM$};
\draw (-1,2) node{$\bullet$} node[above]{$M$};
\draw (1,2) node[above]{$A^+$};
\draw (0,0) node[left]{$A^-$};
\draw (1,-2) node{$\bullet$} node[left]{$T^{-1}M$};
\draw[dashed](3,2)--(1.5,-1);
\draw (4,2) node[above]{$R$};
\end{tikzpicture}
\begin{tikzpicture}[scale=.5]
\fill (0,0)--(2,0)--(1,2)--(0,1)--cycle;
\draw[dashed] (0,0)--++(0,-4);
\draw[dashed] (2,0)--++(4,0);
\draw[dashed] (1,2)--++(-2,4);
\draw[dashed] (0,1)--++(-3,-3);
\end{tikzpicture}

\end{center}
\caption{The outer billiard map outside a quadrilateral.}\label{fig-def-1}
\end{figure}

The dual billiard map outside a regular polygon is known to have bounded orbits since the work of \cite{Ta.95}, \cite{Gutk.Siman.92} or \cite{Beda.13}. 
The symbolic dynamics of these map has begun in \cite{Bed.Ca.11} for the regular pentagon, hexagon and octagon. In \cite{Schwartz.10} Schwartz studied also the dual billiard map outside the regular octagon. In these cases the dual billiard map has a renormalization. We refer also to \cite{Beda.13} for a review on outer billiard.

In Figure \ref{fig4}, we show some pictures of dual billiard map outside regular polygons with $3,5,8$ edges. 
On the left part, we show the cellular decomposition of the plane, in the middle part we restrict to one cone for the regular pentagon, and on the right part we show one ring outside the regular octagon. 
We can remark some resemblances with Figure \ref{fig:tout-8}.
\begin{figure}
\includegraphics[width=4cm]{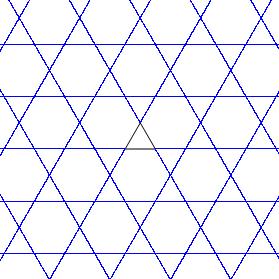}
\includegraphics[width=4cm]{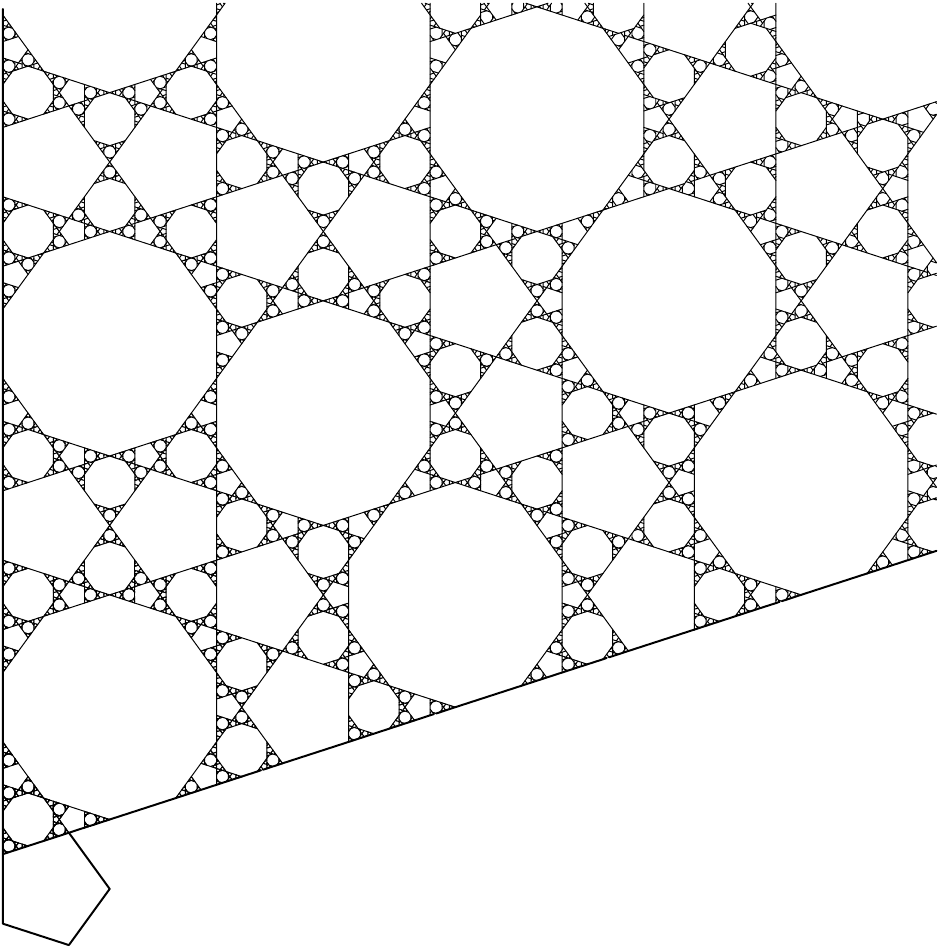}
\includegraphics[width=4cm]{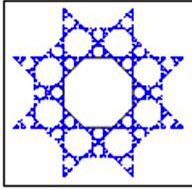}
\caption{Dual billiard dynamics for regular polygons with $3,5$ and $8$ edges. The three pictures represent the cells of the dynamic of the dual billiard map in each case. In the first one, every point has a periodic orbit and the plane is tiled by the cells. In the other cases some fractal sets appear.}\label{fig4}
\end{figure}

To finish this section we recall a map introduced by Schwartz in \cite{Schwartz.10}. The dogbone map is the map $S$ defined by Figure \ref{fig5}. It is a piecewise isometry defined on three sets called $B,C,D$. All the angles in Figure \ref{fig5} are multiple of $\frac{\pi}{4}$. The meaning of the notations will be explained later, see Proposition \ref{lien-dual}.

\begin{figure}
\begin{tikzpicture}[scale=.5]
\draw (-2,0)--(-6.818,0)--++(45:2)--++(0,2)--++(135:2)--++(2,0)--++(45:2)--++(0,2)--++(135:-4.818)--++(-2,0)--++(45:-2)--++(0,-2)--++(135:-2);
\draw (8,0)--(3.182,0)--++(45:2)--++(0,2)--++(135:2)--++(2,0)--++(45:2)--++(0,2)--++(135:-4.818)--++(-2,0)--++(45:-2)--++(0,-2)--++(135:-2);
\draw [dashed] (-3.414,1.414)--(-5.414,3.414);
\draw [dashed] (-3.414,3.414)--++(0,3);

\draw [dashed] (4.586,1.404)--(6.586,3.414);
\draw [dashed] (8,4.818)--++(-3,0);

\draw[->] (0,4)--++(2,0);
\draw (1,3) node{$S$};
\draw (7.5,6) node{$S(C)$};
\draw (5.8,4) node{$S(B)$};
\draw (5.5,1) node{$S(D)$};
\draw (-5,1) node{$C$};
\draw (-4.5,4) node{$B$};
\draw (-2,6) node{$D$};

\end{tikzpicture}
\caption{The definition of the dogbone map.}\label{fig5}
\end{figure}
\section{Results and overview of the paper}
\subsection{Statement}
We consider the angles $\theta\in\{\frac{1}{4},\frac{1}{3},\frac{1}{6},\frac{1}{8}\}$ and give a complete description of the symbolic dynamics for the symmetric maps.

\begin{theorem}\label{thm}
Let $\theta\in\{\frac{1}{4},\frac{1}{3},\frac{1}{6},\frac{1}{8}\}$. Then the language of the bijective symmetric piecewise rotation of angle $\theta$ is substitutive.
\end{theorem}

\begin{remark}
In the cases $\theta\in\{\frac{1}{4},\frac{1}{3},\frac{1}{6}\}$ we can also prove that every orbit is periodic. Indeed it is a consequence from the fact that every orbit is bounded, and by the fact that, in these cases, $T$ is defined by isometries which preserve a lattice.
\end{remark}

The proof is done with Proposition \ref{prop-per} and Proposition \ref{prop-detail}.

\subsection{Overview of the method}
First of all we explain the method of the study :
consider a piecewise rotation with angle $\theta=\frac{p}{q}\in\mathbb Q$. If the map is symmetric, then due to the formulation of the map $T$, the centers of rotations are the points
\begin{equation}
z_0=\frac{e^{2i\pi\theta}}{1-e^{2i\pi\theta}},\quad z_1=\frac{-e^{2i\pi\theta}}{1-e^{2i\pi\theta}}.
\label{centres}
\end{equation}
The imaginary parts of the centers of rotations are equal to $\frac{\pm 1}{2\tan{\pi\theta}}$. In all our cases $\tan{\pi\theta}$ is a positive number, thus each center is a fixed point of $T$ and the codings of these points are $0^\omega$ and $1^\omega$. Consider the cell of one of these periodic words : it is a regular polygon centered at the center of rotation. Depending on the parity of $q$ this polygon has $q$ or $2q$ edges, the length of the side of this polygon is equal to $\tan{(\frac{p\pi}{q})}Im(z_0)$, see Figure \ref{fig=T3} for the case $\theta=\frac{1}{6}$. 
This polygon has one edge on the discontinuity line. 

\begin{definition}
We define the cone $\mathfrak{C}$ as the unique cone such that:
\begin{itemize}
\item its vertex is the vertex of the polygon, on the $\mathbb R$ axis, with the smallest real part. 
\item The cone is included in the upper half plane and delimited by a piece of the discontinuity line $\mathbb R$ and one line supporting a side of the polygon.
\end{itemize}
\end{definition}
An example of cone is given in Figure \ref{fig=T3}.

It is easy to prove the following proposition for every angle $\theta$ in our familly:
See Proposition \ref{prop-per} for a complete proof.

\begin{proposition}
The map $\hat{T}$, first return map of $T$ in $\mathfrak{C}$, is a piecewise isometry. 
The study of $\hat{T}$ is equivalent to the study of $T$. In other terms, there exists a $k$ to one map between the two languages, for some integer $k$.
\end{proposition}
 Thus we will work with the new alphabet denoted by big letters $A, B,\dots$ In each case we will give the map which can transform words on this alphabet on words over the alphabet $\{0,1\}$. We refer to Propositions \ref{prop-per} and \ref{lien-dual} for details depending on the angles. The value of $k$ depends on the initial map $T$.

Now it happens that for an angle in our family we can find a subset and apply Lemma \ref{leminducsub}. Thus we will prove that the map $\hat{T}$ is self-similar, or can be decomposed in different maps which have a renormalization scheme. 
We refer to Propositions \ref{prop-per} and \ref{prop-detail} for details depending on the angles.

\section{Proof of Theorem \ref{thm}}\label{Sec3}
\subsection{Periodic cases}\label{sec-per}
We consider the cases $\theta\in\{\frac{1}{4},\frac{1}{3},\frac{1}{6}\}$.
Let us define the two following substitutions:
$$\sigma_4:\begin{array}{|c|c|c|c|}
\hline
A&B&C&D\\
\hline
DBC&DB&DC&D\\
\hline
\end{array}\quad 
\sigma_6:\begin{array}{|c|c|c|c|c|}
\hline
A&B&C&D&E\\
\hline 
A&AB&AC&ACB&ACCBB\\
\hline
\end{array}
$$

\begin{proposition}\label{prop-per}
The language $\hat{L}$ of the map $\hat{T}$ is the set of factors of the periodic words of the form $z^\omega$ for $z\in Z$, where
\begin{itemize}
\item If $\theta=\frac{1}{4}$, $Z=\displaystyle\bigcup_{n\in\mathbb{N}}\{\sigma_4^n(A),\sigma_4^n(B), \sigma_4^n(C)\}$.
\item If $\theta=\frac{1}{3}$, $Z=\displaystyle\bigcup_{n\in\mathbb{N}}\{A^nB^nC, A^{n+1}B^nC, B^{n+1}A^nC\}$.
\end{itemize}
\end{proposition}
\begin{proof}
$\bullet$ We begin by the proof of the proposition for the angle $\frac{1}{6}$. We refer to Figure \ref{fig=T3}.
In order to obtain the description of the first return map $\hat{T}$ we need to rotate $\mathfrak{C}$ one time around $z_0$ and three times around $z_1$. At this step one piece has come back to $P_0$. We need two other iterations to be sure that every point is coming back to $P_0$. By a simple computation we see that $\mathfrak{C}$ has a partition in five pieces $A, B, C, D, E$ with return words given in the following array by $\begin{array}{|c|c|c|c|c|}
\hline
A&B&C&D&E\\
\hline01^30^2&01^30^3&01^40^2&01^40^4&01^50^4\\
\hline

\end{array}$. 
The triangle $E$ is an equilateral triangle (and the length of the side is equal to the edge of the regular hexagon). The half-lines which define $A$ and $B$ are parallel to one edge of the hexagon, and the edge of $B$ on the discontinuity line has a length equal to $2$. One edge of $C$ is parallel to one edge of $E$. Three edges of $D$ have the same length as the hexagon. 
The restriction of $\hat{T}$ to each of these five sets are some isometries whose vectorial parts are rotations of angle $0,\frac{\pi}{3},\frac{\pi}{3},\frac{2\pi}{3},0$. Remark that those angles are equal to $\frac{n\pi}{3}$ where $n$ is the return time to each piece.

We make an induction on $A$ for the map $\hat{T}$: 
The return map of $\hat{T}$ is conjugated to $\hat{T}$ by a translation of a vector parallel to the discontinuity line. A simple computation gives the different return words to $A$, this defines the substitution $\sigma_6$ as explained in Lemma \ref{leminducsub}. 
We have $$\mathfrak{C}=\displaystyle(\bigcup_{n\in\mathbb{N}}\hat{T}^n(A))\cup C_{B^\omega}\cup C_{C^\omega}\cup C_{D^\omega}\cup C_{E^\omega}\cup C_{(DCB)^\omega}.$$
The words $E^\omega, B^\omega, C^\omega,D^\omega$ and $(DCB)^\omega$ are periodic words of the language. The associated cells are two triangles for $E^\omega$ and $(DBC)^\omega$ and three hexagons for the other infinite words. We deduce the description of the language applying Lemma \ref{leminducsub} to the substitution. Thus the language is a substitutive one.

$\bullet$ The proof for the two other angles is based on the same method.
For $\theta=\frac{1}{4}$ we obtain the following array:
$\begin{array}{|c|c|c|c|}
\hline
A&B&C&D\\
\hline01^30^2&01^30^3&01^20^2&01^20\\ \hline\end{array}$. The associated substitution is $\sigma_4$. Figure \ref{fig-care} describes the map $\hat{T}$, the square has edges of length $1$ and every strip has width $1$. 

$\bullet$ Exactly the same thing can be done for  $\theta=\frac{1}{3}$. The computations are left to the reader.
\end{proof}

\begin{figure}
\begin{tikzpicture}[scale=1]
\draw (-5,0)--(0,0);
\draw (0,0)--(0,5);
\draw (-2,0)--(-2,5);
\draw (-1,0)--(-1,5);
\draw (-1,1)--(0,1);
\draw (-1.5,0.5) node{$C$};
\draw (-.5,0.5) node{$A$};
\draw (-.5,2.5) node{$B$};
\draw (-3.5,0.5) node{$D$};
\draw[->]  (.5,1)--(1.5,1) node[above]{$\hat{T}$};
\draw (2,0)--(7,0);
\draw (7,0)--(7,5);
\draw (2,1)--(7,1);
\draw (2,2)--(7,2);
\draw (6,0)--(6,1);
\draw (6.5,0.5) node{$\hat{T}A$};
\draw (4.5,0.5) node{$\hat{T}C$};
\draw (5.5,1.5) node{$\hat{T}B$};
\draw (6.5,3.5) node{$\hat{T}D$};
\end{tikzpicture}

\vspace{2mm}

\begin{tikzpicture}[scale=.7]
\draw (-5,0)--(0,0);
\draw (-1,0)--(0,0)--(0,1)--(-1,1)--cycle;
\draw (-1,0)--(-1,4);
\draw (-2,0)--(-2,4);
\draw (-3,0)--(-3,4);
\draw (-5,1)--(-1,1);
\draw (-5,2)--(-1,2);
\draw (-.5,.5)node{$\bullet$};
\draw (-1.5,.5)node{$A^\omega$};
\draw (-1.5,1.5)node{$B^\omega$};
\draw (-2.5,.5)node{$C^\omega$};
\end{tikzpicture}
\begin{tikzpicture}[scale=.7]
\draw (-3,0)--(3,0);
\draw (-1,0)--++(0,2);
\draw (0,0)--++(0,1)--++(-1,0)--++(0,-1);
\draw (1,-1)--++(0,1)--++(-1,0)--++(0,-1)--cycle;
\draw (-.5,.5)node{$\bullet$};
\draw (.5,-.5)node{$\bullet$};
\end{tikzpicture}
\caption{The map $\hat{T}$ for a symmetric map of angle $\frac{\pi}{2}$. We also describe the regular polygons, the cone and the substitutive language.}\label{fig-care}
\end{figure}
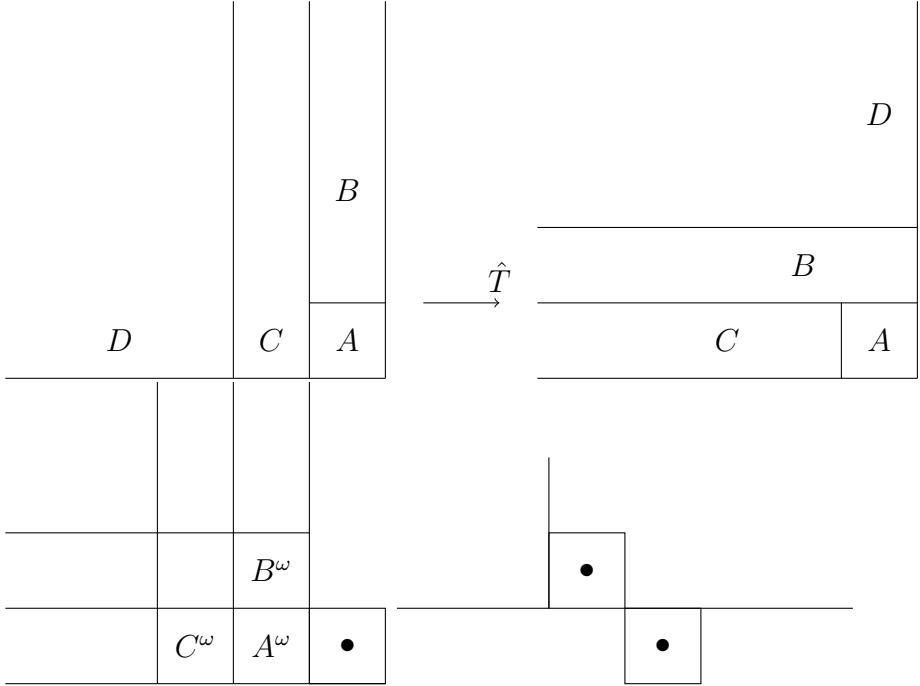

\begin{figure}
\begin{center}
\begin{tikzpicture}[scale=.3]
\draw (3,0)--(12,0);
\draw (12,0)--++(120:8);
\draw (12,0)--++(120:2)--++(60:-2);
\draw (12,0)--++(120:4)--++(0:-6);
\draw (12,0)--++(120:6)--++(0:-6);
\draw (6,0)--(8,3.42);
\draw[<-] (2,1)--(0,1);
\draw (1,2)node{$\hat{T}$};
\draw (-12,0)--(-3,0);
\draw (-3,0)--++(120:8);
\draw (-11,0)--++(120:8);
\draw (-7,0)--++(120:8);
\draw (-3,0)--++(120:2)--++(60:-2);
\draw (-3,0)--++(120:6)--(-8,1.7);
\draw (6,6) node{{\tiny $\hat{T}A$}};
\draw (8,4)node{{\tiny $\hat{T}C$}};
\draw (6,2)node{{\tiny $\hat{T}B$}};
\draw (9,2)node{{\tiny $\hat{T}D$}};
\draw (11,.5)node{{\tiny $\hat{T}E$}};

\draw (-4,.5)node{$E$};
\draw (-6,2)node{$D$};
\draw (-8,4)node{$C$};
\draw (-13,2)node{$A$};
\draw (-9,2)node{$B$};

\draw (-2,-10)--++(120:6);
\draw (-2,-10)--(0,-10)--++(60:2);
\draw (-2,-10)--++(120:2)--++(60:2)--++(2,0)--++(-60:2);

\draw (-10,-10)--(10,-10);
\draw (0,-10)--++(2,0)--++(-60:2)--++(-120:2)--++(-2,0)--++(120:2)--cycle;
\draw (-7,-8)node{$\mathfrak{C}$};
\draw (-1,-8.37)node{$\bullet$};
\draw (1,-11.63)node{$\bullet$};
\end{tikzpicture}

\begin{tikzpicture}[scale=.3]
\draw (-15,0)--(10,-0);
\draw (-2,0)--++(120:2)--++(60:2)--++(2,0)--++(-60:2)--(0,0);
\draw (-6,0)--++(120:2)--++(60:2)--++(2,0)--++(-60:2)--(-4,0);
\draw (-10,0)--++(120:2)--++(60:2)--++(2,0)--++(-60:2)--(-8,0);
\draw (-14,0)--++(120:2)--++(60:2)--++(2,0)--++(-60:2)--(-12,0);

\draw (-8,3.45)--++(120:2)--++(60:2)--++(2,0)--++(-60:2)--(-6,3.45)--(-8,3.45);
\draw (-12,3.45)--++(120:2)--++(60:2)--++(2,0)--++(-60:2)--(-10,3.45)--(-12,3.45);
\draw (-2,0)--++(120:10);

\draw (-1,1.63)node{$\bullet$};
\draw (-3,.5)node{\tiny {$E$}};
\draw (-5,1.5)node{$D^\omega$};
\draw (-9,1.5)node{$B^\omega$};
\draw (-13,1.5)node{$A^\omega$};
\draw (-7,5.5)node{$C^\omega$};
\end{tikzpicture}

\caption{Map $\hat{T}$ for a symmetric map of angle $\frac{\pi}{3}$. The regular polygons, the cone and the substitutive language. In the figure $E$ is an equilateral triangle of size 1. The polygon $D$ is a pentagon with three sides of length one and two of size 2. The bounded side of the cell $B$ is of size 2. All the angles are integer multiples of $\frac{\pi}{3}$.}\label{fig=T3}
\end{center}
\end{figure}
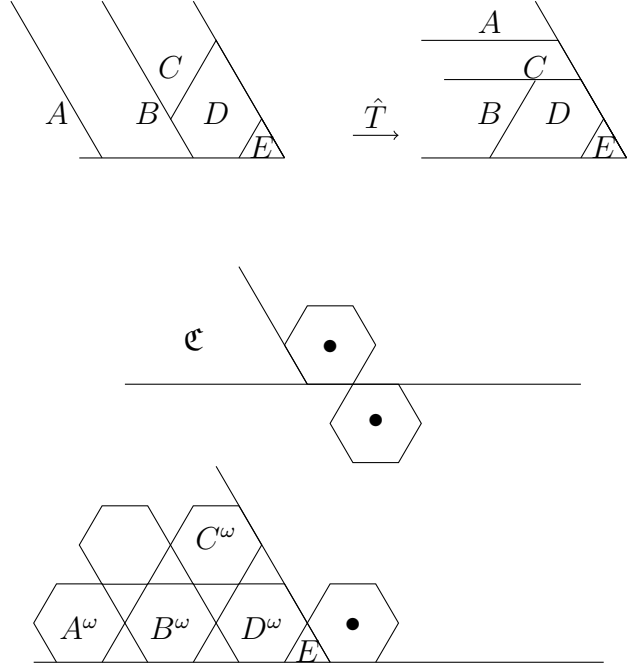

\begin{figure}
\begin{center}
\begin{tikzpicture}[scale=1]
\begin{scope}[shape=circle]
\tikzstyle{every node}=[draw]
\node (q_A) at (0,0) {$\substack{B,C,D\\E,DCB}$};
\end{scope}
\draw[->,shorten >=1pt] (q_A) .. controls +(75:1.2cm) and +(105:1.2cm) .. node[above] {$\sigma_{6}$} (q_A);
\end{tikzpicture}
\end{center}
\caption{Description of the substitutive language for the case $\theta=\frac{1}{6}$.}
\end{figure}
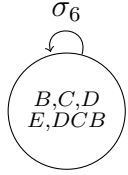

\subsection{Case $\theta=\frac{1}{8}$}\label{18}
We consider the piecewise rotation of the angle $\frac{\pi}{4}$. The case of angle $\frac{2\pi}{5}$ is not treated here since we will see that the study of the first case is closed to the outer billiard outside the regular octagon studied by Schwartz. The symbolic dynamics of the outer billiard outside the regular pentagon has been studied in \cite{Bed.Ca.11} and in \cite{Ta.95}. Using these results, a similar study in the other case can be easily deduced.

We split this part in several lemmas. The proof can be obtained by a direct computation left to the reader. We prefer to show some pictures which describe the dynamics.
\begin{lemma}
The map $\hat{T}$ is a piecewise isometry defined on ten pieces. \end{lemma}
\begin{proof}
 We compute the action of $\hat{T}$ on the set $\mathfrak C$. The iterates $\hat{T}^i\mathfrak{C}, i=1\dots 5$ are all included in the  lower half plane. Then the fifth iterate is split into two parts: one below the real axis and one above. It suffices to iterate these parts to obtain the ten pieces. They are described in Figure \ref{fig=T5} and Figure \ref{octa-ret}.

The first return map of $T$ to the cone $\mathfrak{C}$ is given by Figure \ref{fig=T5}. The pieces are denoted by letters $A,\dots, J$. It is defined on five pieces, three are unbounded and denoted $A, B, C$ and two are compact and denoted by $D$ and $\mathfrak{C}_1$. The set $\mathfrak{C}_1$ is the black set in Figure \ref{fig=T5}. There is a partition of $\mathfrak{C}_1$ in six subsets denoted $E,F ,\dots,J$ according to the different return times of points to $\mathfrak{C}$, see Figure \ref{octa-ret}. The list of the return words is given in the following array:
$$\begin{array}{|c|c|c|c|c|c|c|c|c|c|}
\hline
A&B&C&D&E&F&G&H&I&J\\
\hline
 01^40^3&01^50^3&01^40^4&01^50^4&01^50^5&01^60^5&01^70^6&01^60^4&01^70^5&01^60^6\\
 \hline
\end{array}
$$
As explained in the proof of Proposition \ref{prop-per}, the restriction of $\hat{T}$ to each of the ten sets is a rotation of an angle multiple of $\frac{\pi}{4}$. This multiple is equal to the length of the return word. 

\end{proof}

\begin{figure}[h]
\begin{center}
\begin{tikzpicture}[scale=.4]
\draw (-17,0)--(0,0)--++(135:15);
\draw (-2.818,0)--++(45:1);
\fill (-4.818,0)--(-3.414,1.414)--++(0,2)--(0,0)--cycle;
\draw (-6.818,0)--++(135:10);
\draw (-6.818,0)--++(135:2)--++(0,6.818);
\draw (-13.636,0)--++(135:10);
\draw (-17,2) node{$A$};
\draw (-12,2) node{$C$};
\draw (-6,4) node{$D$};
\draw (-10,7) node{$B$};

\draw (-17,-14)--(0,-14)--++(135:16);
\fill (-4.818,-14)--(-3.414,-12.586)--++(0,2)--(0,-14)--cycle;
\draw (0,-14)--++(135:6.818)--++(-12,0);
\draw (0,-14)--++(135:13.26)--++(-10,0);
\draw (0,-14)--++(135:6.818)--++(-2,0)--++(45:-6.818);
\draw (-13,-12) node{$\hat{T}C$};
\draw (-10,-8) node{$\hat{T}B$};
\draw (-13,-3) node{$\hat{T}A$};
\draw (-8,-12) node{$\hat{T}D$};
\draw (2,-2) node{$\hat T$};
\draw[->](3,0)--(3,-3);
\end{tikzpicture}
\caption{The map $\hat{T}$ associated to the angle $\frac{\pi}{4}$ and $\sigma=0$. 
The coordinates of points and length of sides are:
All the angles are integer multiple of $\frac{\pi}{4}$.
In the black set two sides have lengths $2+\sqrt{2}$. The two smaller ones have length 1.
The bounded sides of $C$ and $B$ have the same length $1+\sqrt{2}$. In the cell $D$ four edges have the same length 1, the other ones have the same length $2+\sqrt 2$.
}\label{fig=T5}
\end{center}
\end{figure}

\begin{corollary}
 There exists an invariant compact set for $\hat{T}$, denoted $\mathfrak{C}_1$.
\end{corollary}
\begin{proof}
This set is the black set in Figure \ref{fig=T5}. We have $$\mathfrak{C}_1=E\cup F\cup G\cup H\cup I\cup J.$$ 
This set is globally invariant due to the action of $\hat{T}$ described in Figure \ref{octa-ret}.
\end{proof}
\begin{figure}[h!]
\begin{center}
\begin{tikzpicture}[scale=.7]
\draw (0,0)--++(45:5);
\draw (0,0)--(7,0);
\draw (7,0)--++(135:7.1);
\draw (5,0)--(5,2);
\draw (3,0)--++(135:2.14);
\draw (6,0)--++(45:.71);
\draw (2,2)--++(0,3);
\draw (2.8,2.8)--++(1.4,0);

\draw[->] (-3,2)--(-1,2);

\draw (-10,0)--++(45:2.8);
\draw (-10,0)--(-3,0);
\draw (-3,0)--++(135:7.1)--++(0,-5);
\draw (-8,3)--++(2,0);
\draw (-7,0)--++(135:1.41);
\draw (-6,0)--++(45:2.1);
\draw (-4,0)--++(0,1);

\draw (1,.5) node{{\tiny $\hat{T}E$}};
\draw (3,1) node{{\tiny $\hat{T}F$}};
\draw (5.6,1) node{{\tiny $\hat{T}G$}};
\draw (2.8,3.5) node{{\tiny $\hat{T}H$}};
\draw (3.5,3.1) node{{\tiny $\hat{T}I$}};

\draw (-9,.5) node{$E$};
\draw (-7,1) node{$F$};
\draw (-5,.5) node{$G$};
\draw (-7,3.5) node{$H$};
\draw (-3.6,.3) node{$I$};

\draw (-2,2.5) node{$\hat T$};
\end{tikzpicture}
\end{center}
\caption{Dynamics inside $\mathfrak{C}_1$: the black compact set in Figure \ref{fig=T5}.
The cells $E$ and $H$ are isometric triangles of sides $1,1,\sqrt{2}$.
They are also isometric to the union of the cells of $G$ and $I$. Two sides of $G$ have the same lengths as two sides of $I$.
In the cell $F$ three sides have for lengths $1$. A simple computation allows us to obtain the other lengths.}\label{octa-ret}
\end{figure}

Denote by $\mathfrak{C}_3$ the orbit of the set $A$ under $\hat{T}$.
\begin{lemma}\label{pfffff} 
The set $\mathfrak{C}_3$ is $\hat{T}$-invariant. The first return map by $\hat{T}$ to $A$ is self-similar.
\end{lemma}

\begin{figure}
\begin{center}
\begin{tikzpicture}[scale=.3]
\draw (-20,0)--(0,0)--(-15,15);
\fill (0,0)--(-3.5,0)--(-2.5,1)--(-2.5,2.5);
\draw (-5,0)--(-15,10);
\draw (-10,0)--(-18,8);
\draw (-3.5,3.5)--(-18,3.5);
\draw (-7,7)--(-18,7);
\draw (-6,1)--(-6,6);
\draw (-8.5,0)--(-5,3.5);
\draw (-13.5,0)--(-10,3.5);
\draw (-18.5,0)--(-15,3.5);
\draw (-9.5,3.5)--(-9.5,6);
\draw (-12,3.5)--(-8.5,7);
\draw (-11,0)--(-11,2.5);
\fill[color=gray, opacity=.5] (-20,0)--(-10,0)--(-11,1)--(-11,2.5)--(-10,3.5)--(-8.5,3.5)--(-9.5,4.5)--(-9.5,6)--(-8.5,7)--(-7,7)--(-15,15);
\end{tikzpicture}
\caption{{\it The set $\mathfrak{C}_3$ is colored in grey: it is an union of the pieces of the return map to $A$.}}\label{fig=ret=14=A}
\end{center}
\end{figure}

\begin{proof}
A simple computation, see Figure \ref{fig=ret=14=A}, shows that it is self similar to the initial map $\hat{T}$, and the associated substitution $ \sigma_{8,3}$ is given by 
$$
\sigma_{8,3}:
\begin{cases} \begin{array}{|c|c|c|c|c|c|}
\hline
A&B&C&D&E&F\\
\hline
A&AB&AC&ABC&ABC^2A&AB^2CB\\
\hline
\end{array}\\


\\

\begin{array}{|c|c|c|c|}
\hline
G&H&I&J\\
\hline
AB^4C^2&AB^2CA&AB^4CA&AB^2CBC\\
\hline
\end{array}
\end{cases}
$$
\end{proof}

Now we define $$\mathfrak{C}=\mathfrak{C}_1\cup \mathfrak{C}_2\cup \mathfrak{C}_3.$$
The set $\mathfrak{C}_2$ is the complement in $\mathfrak{C}$ of the two other sets.
\begin{lemma}
The set $\mathfrak{C}_2$ is invariant under $\hat{T}$. Consider the complement in $\mathfrak{C}_2$ of the three big regular octagons. The restriction of $\hat{T}$ to this set is the dogbone map $S$.
\end{lemma}
\begin{proof}
The definition of $\mathfrak{C}_2$ and the fact that the two other sets $\mathfrak{C}_1, \mathfrak{C}_3$ are invariant prove the invariance. It is clear that the restriction of $\hat{T}$ is a bijective map. This part is an exchange of three pieces. The pieces are denoted by the same letters $B, C$ and $D$ as previously since they are restrictions of the initial pieces.
This map has a dynamics given in Figure \ref{octa-ret2}:
there are three big regular octagons corresponding to the periodic orbits: $B^\omega, C^\omega$ and $D^\omega$. 
Outside these octagons, the dynamics is an exchange of three pieces.
As the map defined in the set $\mathfrak{C}_1$, this map has been studied in \cite{Schwartz.10} (Section $7.3$ of Arxiv version) 
 where the shape is called a dogbone. All the periodic cells are regular octagons. 
\end{proof}

\begin{figure}

\begin{tikzpicture}[scale=.5]
\draw (2,0)--(-6.818,0)--++(135:2)--++(0,2)--++(45:2)--++(2,0)--++(135:2)--++(0,2)--++(45:2)--++(2,0)--++(135:-8.818);
\draw (2,0)--++(45:2)--++(0,2);
\draw (0,0)--++(135:2)--++(0,6.818);
\draw (0,0)--++(135:6.818);

\draw[->] (3,5)--(5,5);

\draw (17,0)--(8.182,0)--++(135:2)--++(0,2)--++(45:2)--++(2,0)--++(135:2)--++(0,2)--++(45:2)--++(2,0)--++(135:-8.818);
\draw (17,0)--++(45:2)--++(0,2)--++(135:2)--++(-8.818,0);
\draw (10.182,0)--++(45:6.818); 

\draw (-5,2) node{$C$};
\draw (-4,6) node{$B$};
\draw (0,2) node{$D$};

\draw (13,7) node{$\hat{T}B$};
\draw (10,2) node{$\hat{T}C$};
\draw (15,2) node{$\hat{T}D$};
\draw (4,5.2) node[above]{$\hat T$};
\end{tikzpicture}

\caption{ Definition of the dynamics inside $\mathfrak{C}_2$ and link with the Dogbone map.}\label{octa-ret2}
\end{figure}

The preceding discussion can be summarized in the following proposition.

\begin{proposition}\label{lien-dual}
The dynamics of $\hat{T}$ is given by three maps: the renormalization map defined by Lemma \ref{pfffff}, the dynamics of $S$ and the dynamics of $\hat{T}$ restricted to $\mathfrak{C}_1$.
\end{proposition}
\begin{proof}
Consider $z\in\mathfrak{C}$. There are three cases: either $z$ belongs to  $\mathfrak{C}_1$, either it belongs to $\mathfrak{C}_2$ or to $\mathfrak{C}_3$.
\begin{itemize}
\item First case: the result is obvious. 
\item If $z$ belongs to $\mathfrak{C}_2$ there are two cases: either it is a periodic point of code $B^\omega, C^\omega$ or $D^\omega$ or its orbit is described by $S$. 
\item The last case is if $z$ belongs to $\mathfrak{C}_3$. Then Lemma \ref{pfffff} implies that the renormalization process can be applied.
\end{itemize}
\end{proof}

As a by product we deduce:
\begin{proposition}\label{prop-detail}
 The language $\hat{L}$ of $\hat{T}$ is substitutive. It is the set of factors of the periodic words of the form $z^\omega$, where $z\in Z$. The set $Z$ is obtained as a substitutive language.  
 \end{proposition}
\begin{proof}



$\bullet$ By Lemma \ref{pfffff} the coding of an orbit of an element of $\mathfrak{C}_3$ will be the image under $\sigma_{8,3}$ of the coding of the orbit of a point in the two compact sets $\mathfrak{C}_1$ and $\mathfrak{C}_2$.

$\bullet$ Now we study the map $\hat{T}$ restricted to $\mathfrak{C}_1$, see Figure \ref{octa-ret}.  Applying Lemma \ref{leminducsub} and the results of Schwartz, the symbolic dynamics of this map is thus ruled by a substitutive scheme denoted $\sigma_{8,1}$. We make a short description of this dynamics, we refer to \cite{Schwartz.10} for the details and to Figure \ref{fig:split:8}:
All the periodic cells are regular octagons of different sizes: First of all, the easiest periodic words are $E^\omega,F^\omega,G^\omega$ and $H^\omega$. Their cells are regular octagons inscribed in the polygons $E,F,G$ and $H$ as seen on the figure.
Then, there is a periodic orbit of period $E^2JIH^2FGF$. In Figure \ref{fig:split:8}, it corresponds the nine regular octagons closed to the biggest ones. The next level is made of three different orbits made by octagons of the same size: $E^2FG^2J, FH^2IG^2$ and $E^5FGEH^5JI$. 

Finally the language of the dynamics restricted to this invariant set is given by the images of these words by $\sigma_{8,1}$. We obtain a substitutive language made by periodic words of the form $z^\omega, z\in Z_1$ with
$$Z_1=\bigcup_{n\in\mathbb{N}}\{\sigma_{8,1}^n(E),\sigma_{8,1}^n(F),\sigma_{8,1}^n(G),\sigma_{8,1}^n(H)\}.$$

$\bullet$ To finish, we study the map $\hat{T}$ restricted to $\mathfrak{C}_2$. 
The language is also a substitutive language with a substitution denoted $\sigma_{8,2}$. Once again we refer to the work of Schwartz for more details and a description of the substitutions.
Finally applying Lemma \ref{leminducsub}, we deduce that in this compact set the language is made by periodic words of the form $z^\omega, z\in Z_2$ with
 $$Z_2=\sigma_{8,2}(Z_3)\cup \{C\}\cup\{D\}\cup\{B\} \quad\text{where}$$
$$ Z_3=\bigcup_{k\in\mathbb{N}}\{\sigma_{8,2}^k(C), \sigma_{8,2}^k(D)\}.$$

$\bullet$ To resume we have that the total language is made  by periodic words of the form $z^\omega, z\in Z$ with $$Z=\bigcup_{m\in\mathbb{N}}\bigcup_{z\in Z_1\cup Z_2}\{\sigma_{8,3}^m(z)\}.$$
Thus the substitution $\sigma_{8,3}$ allows to pass from the compact invariants sets to the whole space.
\end{proof}

All this discussion can be resumed by the following graph, which describes the substitutive language.

\begin{tikzpicture}
\begin{scope}[shape=circle,minimum size=.6cm]
\tikzstyle{every node}=[draw]
\node (q_A) at (0,0) {$A,\dots, J$};
\node (q_E) at (-6,0) {$B,C,D$};
\node (q_1) at (3,0) {$E,F,G,H$};
\node (q_2) at (-3,0) {$B,C,D$};
\end{scope}
\draw[->] (q_1) -- (q_A) ;
\draw (1.3,0) node[above]{$\sigma_{8,3}$};
\draw (-1.3,0) node[above]{$\sigma_{8,3}$};
\draw[->] (q_E)--(q_2);
\draw (-4.5,0) node[below]{$\sigma_{8,2}$};
\draw[->] (q_2) -- (q_A);
\draw[->,shorten >=2pt] (q_A) .. controls +(75:1.4cm) and +(105:1.4cm) .. node[above] {$\sigma_{8,3}$} (q_A);
\draw[->,shorten >=2pt] (q_E) .. controls +(75:1.4cm) and +(105:1.4cm) .. node[above] {$\sigma_{8,2}$} (q_E);
\draw[->,shorten >=2pt] (q_1) .. controls +(75:1.4cm) and +(105:1.4cm) .. node[above] {$\sigma_{8,1}$} (q_1);
\end{tikzpicture}
\begin{remark}
The dynamics is given by a substitutive system on a quite big alphabet. It seems more complicated than staying on the alphabet $\{0,1\}$. Nevertheless it appears to be the best way in order to describe the dynamics. Moreover there exists a morphism to pass to the initial alphabet.
We left to the reader the passage to the initial alphabet.
\end{remark}

\begin{remark}
With this description, it is easy to obtain an explicit example of a bounded and non periodic orbit. It suffices to iterate the substitution $\sigma_{8,1}$ on one letter. The limit is an infinite aperiodic word. It can not be the coding of a periodic orbit.
\end{remark}

\subsection{Comparison with the Theorem of Goetz-Quas}
We show on two examples how our result are related to \cite{Goet.Quas.09}. 
\begin{itemize}
\item First consider $\theta=\frac{1}{4}$. By Proposition \ref{prop-per} the following periodic words are in the language:
$A^\omega,B^\omega$ and $C^\omega$. The change of alphabet maps these words respectively to the periodic words $(0^31^3)^\omega, (0^21^2)^\omega$ and $(0^31^2)^\omega$. The same method applies for the word of period $\sigma_4(A)=DBC$. We obtain a periodic word of period $0^31^20^21^30^21^20^3$. All these words correspond to some rationally coded points associated to the sequence $(\frac{p_k}{q_k}=\frac{1}{4k+1})_{k\geq 0}$. In this case the language is totally described by these words.

\item Now consider the case $\theta=\frac{1}{8}$. The result of \cite{Goet.Quas.09} implies that the rotationally coded points form cells of one type: a regular polygon with $8$ edges. The periodic islands surrounding the points on this orbit touch so that the union forms an invariant annulus surrounding the origin. Each annulus correspond to one or two periodic words.
 In Figure \ref{fig:tout-8} we can see the rings of regular octagons. The first ring is made of 10 polygons, and the second of $18$ polygons. In our description, the first ring corresponds to the regular octagon inside $D$, which is invariant by $\hat{T}$ in $\mathfrak{C}_3$. It is coded by $D^\omega$ and thus also by $(0^51^5)^\omega$. The other ring corresponds to two periodic orbits coded by $(0^41^5)^\omega$ and $(0^51^4)^\omega$. In our description they correspond to the words $B$ and $C$. If we look at the third ring the periodic word is the image of $D$ by $\sigma_{8,1}$. Finally look at the regular octagons of smaller size and their first ring. It is coded by $(0^51^6)^\omega$. In our description it is coded by $F^\omega$.
  In other terms the cells of Figure \ref{fig:tout-8} can be splitted in the big octagons associated to the  rotationally coded points and the two dynamics inside $\mathfrak{C_3}$ and $\mathfrak{C_2}$ shown in Figure \ref{fig:split:8}. Thus we see that the words described in \cite{Goet.Quas.09} appear in our description of the total language. Nevertheless they do not represent all the infinite words since there exists non periodic infinite words. 
\end{itemize}

\begin{figure}
\includegraphics[width=6cm]{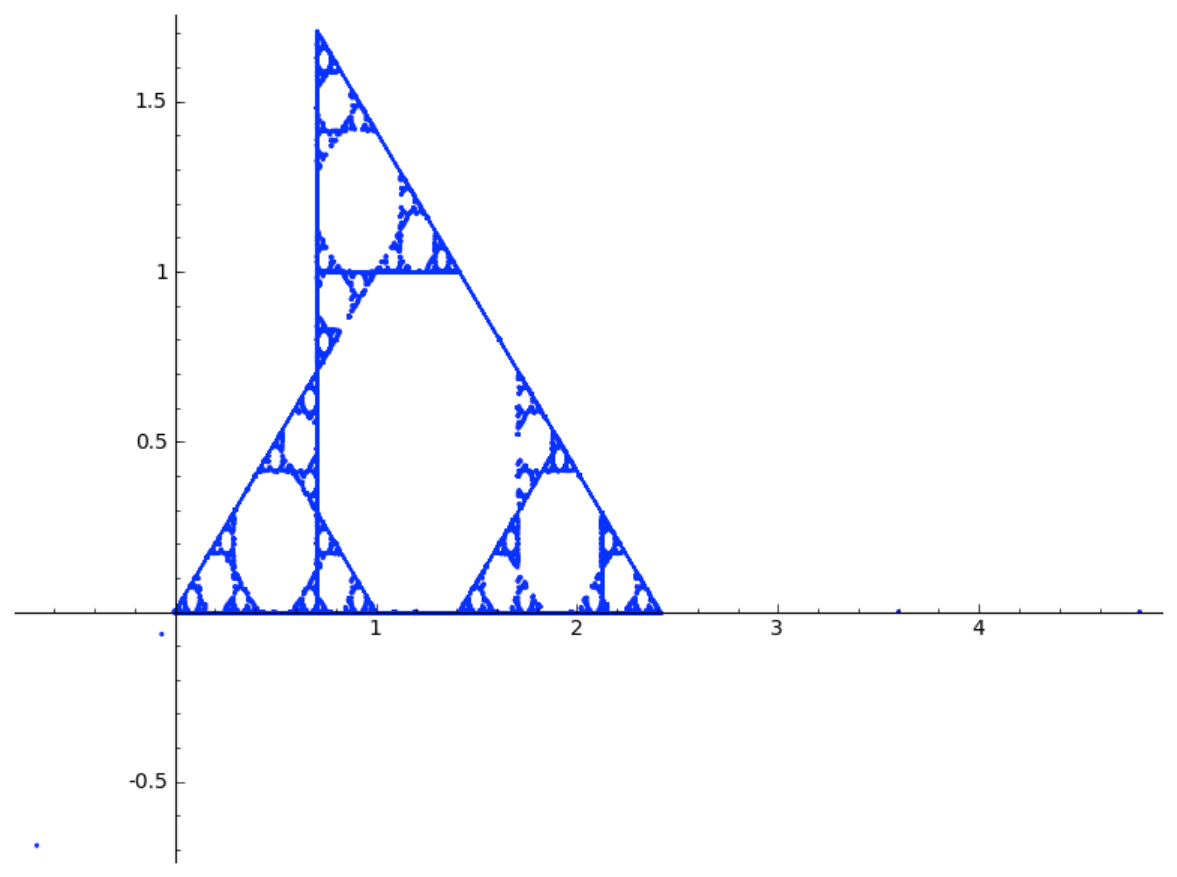}
\includegraphics[width=5cm]{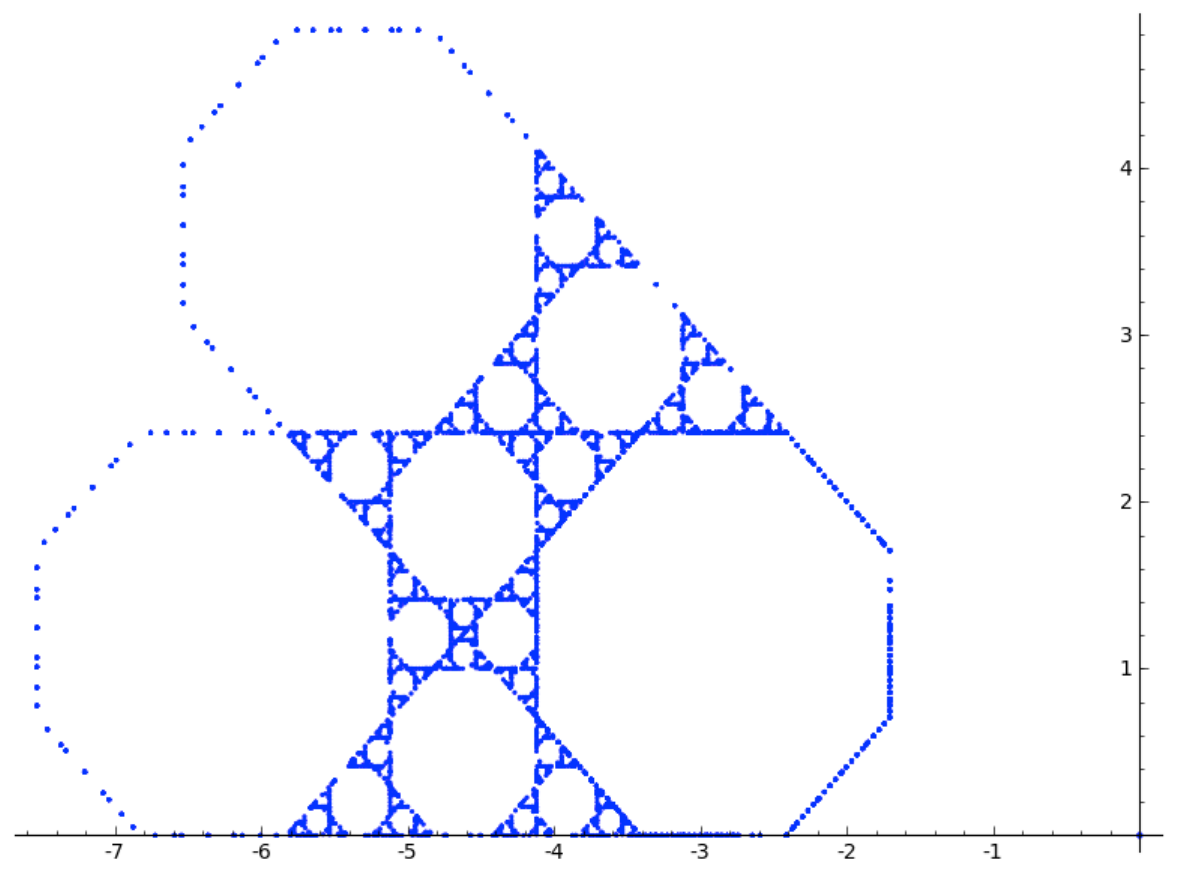}
\caption{Decomposition of the dynamics. On the left there is the dynamics inside $\mathfrak{C}_1$ and on the right the dynamics inside $\mathfrak{C}_2$ of the dogbone map.}\label{fig:split:8}
\end{figure}

\bibliographystyle{plain}
\bibliography{biblio-rot-v2}

\begin{thebibliography}{10}

\bibitem{Ad.Ki.Tr.01}
R.~Adler, B.~Kitchens, and C.~Tresser.
\newblock Dynamics of non-ergodic piecewise affine maps of the torus.
\newblock {\em Ergodic Theory Dynam. Systems}, 21(4):959--999, 2001.

\bibitem{As.Go.04}
P.~Ashwin and A.~Goetz.
\newblock Polygonal invariant curves for a planar piecewise isometry.
\newblock {\em Transaction of the American Mathematical Society}, 2004.

\bibitem{Bert.Dele.14}
V.~Berth\'{e} and V.~Delecroix.
\newblock Beyond substitutive dynamical systems: {$S$}-adic expansions.
\newblock In {\em Numeration and substitution 2012}, RIMS K\^{o}ky\^{u}roku
  Bessatsu, B46, pages 81--123. Res. Inst. Math. Sci. (RIMS), Kyoto, 2014.

\bibitem{Bosh.Goet.03}
M.~Boshernitzan and A.~Goetz.
\newblock A dichotomy for a two-parameter piecewise rotation.
\newblock {\em Ergodic Theory Dynam. Systems}, 23(3):759--770, 2003.

\bibitem{Buzz.01}
J.~Buzzi.
\newblock Piecewise isometries have zero topological entropy.
\newblock {\em Ergodic Theory Dynam. Systems}, 21(5):1371--1377, 2001.

\bibitem{Beda.13}
N.~Bédaride.
\newblock A characterization of quasi-rational polygons.
\newblock {\em Nonlinearity}, 25(11):3099--3110, 2012.

\bibitem{Bed.Ca.11}
N.~Bédaride and J.~Cassaigne.
\newblock Outer billiard outside regular polygons.
\newblock {\em J. Lond. Math. Soc. (2)}, 84(2):303--324, 2011.

\bibitem{BK2}
N.~Bédaride and I.~Kaboré.
\newblock Symbolic dynamics of a piecewise rotation: case of the non symmetric
  bijective maps.
\newblock {\em Qual. Theory Dyn. Syst.}, 17(3):651--664, 2018.

\bibitem{Che.Goe.Qua.12}
Y.~Cheung, A.~Goetz, and A.~Quas.
\newblock Piecewise isometries, uniform distribution and {$3\log 2-\pi^2/8$}.
\newblock {\em Ergodic Theory Dynam. Systems}, 32(6):1862--1888, 2012.

\bibitem{Goet.98}
A.~Goetz.
\newblock Dynamics of a piecewise rotation.
\newblock {\em Discrete Contin. Dynam. Systems}, 4(4):593--608, 1998.

\bibitem{Goet.Quas.09}
A.~Goetz and A.~Quas.
\newblock Global properties of a family of piecewise isometries.
\newblock {\em Ergodic Theory Dynam. Systems}, 29(2):545--568, 2009.

\bibitem{Gutk.Siman.92}
E.~Gutkin and N.~Sim{\'a}nyi.
\newblock Dual polygonal billiards and necklace dynamics.
\newblock {\em Comm. Math. Phys.}, 143(3):431--449, 1992.

\bibitem{Schw.09}
R.~E. Schwartz.
\newblock {\em Outer billiards on kites}, volume 171 of {\em Annals of
  Mathematics Studies}.
\newblock Princeton University Press, Princeton, NJ, 2009.

\bibitem{Schwartz.12}
R.~E. Schwartz.
\newblock Square turning maps and their compactifications.
\newblock {\em Geom. Dedicata}, 192:295--325, 2018.

\bibitem{Schwartz.10}
R.~E. Schwartz.
\newblock {\em The plaid model}, volume 198 of {\em Annals of Mathematics
  Studies}.
\newblock Princeton University Press, Princeton, NJ, 2019.

\bibitem{Ta.95}
S.~Tabachnikov.
\newblock On the dual billiard problem.
\newblock {\em Adv. Math.}, 115(2):221--249, 1995.

\end{thebibliography}



\end{document}